\documentclass{amsart}

\usepackage{hyperref}
\usepackage{amsmath,amsthm,amssymb,verbatim,mathrsfs}%
\usepackage{color}
\usepackage{url}

\newtheorem{thm}{Theorem}[section]
\newtheorem{cor}[thm]{Corollary}
\newtheorem{lem}[thm]{Lemma}
\newtheorem{prop}[thm]{Proposition}

\theoremstyle{definition}

\theoremstyle{definition}

\theoremstyle{remark}
\newtheorem{rem}[thm]{Remark}

\newcommand\be{\begin{equation}}
\newcommand\ee{\end{equation}}
\newcommand\bee{\begin{equation*}}
\newcommand\eee{\end{equation*}}
\newcommand\ben{\begin{enumerate}}
\newcommand\een{\end{enumerate}}

\def\Q{\ensuremath {{\mathbb Q}}}

\def\CC{\ensuremath {{\mathbb C}}}
\def\H{\ensuremath {{\mathcal H}}}

\def\R{\ensuremath {{\mathbb R}}}

\def\A{\ensuremath {{\mathbb A}}}
\def\AA{\ensuremath {{\mathcal A}}}

\def\Z{\ensuremath {{\mathbb Z}}}

\def\C{\ensuremath {\mathscr C}}

\def\g{\ensuremath {\mathfrak g}}

\def\c{\ensuremath {\mathfrak c}}

\def\a{\ensuremath {\mathfrak a}}
\def\tr{\ensuremath {\mathrm{tr}}}
\def\GL{\ensuremath {\mathrm{GL}}}

\def\vol{\ensuremath {\mathrm{vol}}}
\def\bs{\ensuremath {\backslash}}

\numberwithin{equation}{section}

\title{A note on  Poisson summation for GL(2)}
\author{Tian An Wong}
\email{tiananw@umich.edu}
\address{University of Michigan, Dearborn, MI}


\begin{document}

\begin{abstract}
Using analytic number theory techniques, Altu\u{g} showed that the contribution of the trivial representation to the Arthur-Selberg trace formula  for GL(2) over $\Q$ could be cancelled by applying a modified Poisson summation formula to the regular elliptic contribution. Drawing on recent works, we re-examine these methods from an adelic perspective. 
\end{abstract}

\maketitle

\tableofcontents \addtocontents{toc}{\protect\setcounter{tocdepth}{1}}

\section{Introduction}

\subsection{Motivation}

The Arthur-Selberg trace formula is a central tool in the study of automorphic forms. A problem that arose in Langlands' later studies of the trace formula was the removal of the contribution of nontempered representations to the spectral expansion. In \cite{FLN}, Frenkel, Langlands, and Ng\^o proposed that an additive Poisson summation might be applied to the elliptic part of the stable trace formula and that the zeroth term of the dual sum would be equal to the character of the trivial representation of $G$. 

Consider the case of $G=\GL(2)$ over a number field $F$. The regular elliptic part of the trace formula for $G$ for a suitable test function $f$ is given by
 \[
J^G_\mathrm{ell}(f)= \sum_\gamma a^G(\gamma)O_\gamma(f) =  \sum_{\gamma}\text{vol}(G_\gamma(F)\backslash G_\gamma(\A)^1) \int_{G_\gamma(\A)\backslash G(\A)} f(x^{-1} \gamma x) dx,
 \]
where the sum runs over (stable) regular elliptic conjugacy classes of $G(F)$, and $G_\gamma$ the centralizer of $\gamma$ in $G$. The idea of \cite{FLN}, in this case, is first to reindex the conjugacy class $\gamma$ in terms of a certain Steinberg-Hitchin base. Its naive form is simply the space of characteristic polynomials $(\tr(\gamma),\det(\gamma)) \in F\times F^\times\subset \A\times \A^\times$, where the summands are then to be interpreted suitably as factorizable functions over $\A\times \A^\times$ so that an additive Poisson summation may be applied to the $F$ sum. There are several obstructions that arise:
\ben
\item
The volume term is a global object, a priori defined over $F$ and not adelically. In all the related works \cite{FLN,ST,alt1,melk,beram}, it is re-expressed as an $L$-value using a class number formula for tori, which can then be factorized. (This is what Frenkel, Langlands, and Ng\^o refer to as the adelization of the trace formula.)
\item
The sum over $F$ is incomplete, running only over the image of regular elliptic elements. If we naively complete the sum while the $L$-value expression above, the $L$-value diverges at the image of hyperbolic elements.
\item
The orbital integral has singularities, which must be controlled in order to apply an appropriate Poisson formula to the summands.
\een
Succeeding this, heuristics of \cite{FLN} suggest that the dominant term in the dual Poisson sum should equal the contribution of the trivial representation of $G$.

For a special case of $\GL_2(\Q)$ with a certain family of test functions, these problems were overcome by Altu\u{g} \cite{alt1} using classical analytic number theory techniques, completing a test case initiated by Langlands in \cite{BE}. In particular, an approximate functional equation was applied to the $L$-value, which allowed a proper completion of the sum and the smoothing of the orbital integral's (jump) singularities. This was recently generalized to totally real fields in \cite{melk} (with mild assumptions in characteristic 2) and over $\mathbb Q$ but allowing for ramification at finite primes \cite{beram}. The demanding analysis required in these works strongly indicate that an adelic reformulation is necessary for further development. (Also, \cite{GKM} showed that in higher rank the class number formula approach encounters Artin's conjecture for $L$-functions as an obstruction.)

The purpose of this partly expository note is to provide an adelic analogue of \cite{alt1,melk}, re-establishing the cancellation of the contribution of the trivial representation as envisioned in \cite{FLN} for some simple cases related to GL(2). We give two adelic treatments: the first makes use of the analysis carried out by Langlands \cite{ST} and Gordon \cite{gornus} regarding the behaviour of stable orbital integrals over the Steinberg-Hitchin base, while the second uses a relation to binary quadratic forms following Matz \cite{matz0}. 

\subsection{Poisson apr\`es Langlands}

Let $Z$ be a subgroup of GL(1) and $G(\A)^Z = \{g \in G(\A): \det(g) \in Z(\A)\}$. Also let $K$ be a compact open subgroup of $G(\A_f)$ where $\A_f$ is the finite adele ring. 
 Let $\C(G(\A)^Z;K)$ be the space of smooth bi-$K$-invariant functions on $G(\A)^Z$ which belong to $L^1(G(\A)^Z)$ along with all their derivatives. It is a Fr\'echet space under the family of seminorms $|| f * X||_{1}$, where $X$ is any element in the universal enveloping algebra of the Lie algebra of $G(\R)\cap G(\A)^Z$ and $f\in \C(G(\A)^Z;K)$. The Hecke subalgebra of smooth compactly supported functions $\H(G,K) = C_c^\infty(G(\A)^Z;K)$ is dense in this space. Finally, write $\C(G^Z)$ for the inductive limit of $\C(G(\A)^Z;K)$ as $K$ ranges over open compact subgroups of $G(\A)^Z$. This is a subset of test functions that Finis, Lapid and M\"uller extended the trace formula to \cite{FLM,FL}. Also, given a localization $F_v$ of $F$, we denote the analogous spaces by $\C(G(F_v)^Z)$ and so on when the context is clear. 

Let  $\AA = \mathbb G_a\times \mathbb G_m$. We denote by $\AA_\text{ell}$ the elliptic locus, $\AA_\text{rss}$ the regular semisimple locus, and $\AA_\text{sing}$ the singular locus respectively. Also set $\AA_\text{spl} = \AA_\text{rss} - \AA_\text{ell}$ and $\AA_\text{nell} = \AA - \AA_\text{ell}$ so that 
\[
\AA = \AA_\text{rss} \cup \AA_\text{sing} = \AA_\text{ell} \cup \AA_\text{spl} \cup \AA_\text{sing} = \AA_\text{ell}\cup \AA_\text{nell}.
\]
Given $f$ on $G$, we obtain a function $\theta_f$ on $\AA$ as the pushforward given by integrating along the fiber over $a\in \AA$ as in \eqref{thetaf}, and denote by $\hat\theta_f$ its Fourier transform in the $\mathbb G_a$ factor. In this paper, we shall assume $Z$ to be a finite group: this includes the cases of SL(2) in \cite{ST} and the determinant condition imposed in \cite{alt1,melk,beram}. 

\begin{thm}
\label{main}
Let $f \in \C(G^Z)$. Then the elliptic part of the trace formula for $\GL(2)$ can be expressed as 
\[
J^G_\mathrm{ell}(f) = \tr(1(f)) + \sum_{\substack{a\in \AA(F)\\ a\neq 0}}  \hat \theta_f(a) - \sum_{a\in \AA_\textnormal{nell}(F)} \theta_f(a).
\]
\end{thm}
 
\noindent We prove the theorem in two steps. We first assume in Section \ref{aell} that $f$ belongs to the space $C_c^\infty(G^Z)$ of smooth compactly supported functions on $G(\A)^Z$, so that we can freely use the earlier computations of Langlands \cite{ST}. Its proof amounts to piecing together the known results and experience gained in the works discussed above. Then in Section \ref{cont}, following the method of Finis and Lapid, we extend this continuously to $\C(G^Z)$.

\subsection{Relation to the $r$-trace formula}
\label{big}

Before continuing, let us describe the larger context in which this work is situated. The Arthur-Selberg trace formula remains the most powerful tool by which general cases of Langlands' Functoriality conjecture have been proved. In the original Beyond Endoscopy proposal \cite{BE}, Langlands proposed a new refinement of the trace formula in order to attack Functoriality in full generality. The idea was to weight the spectral side of the trace formula with coefficients that would detect when a given automorphic representation $\pi$ was a transfer from a smaller group. This coefficient was to be the order of the pole at $s=1$ of the associated $L$-function $L(s,\pi,r)$. By a standard Tauberian theorem often used in analytic number theory, Langlands gave an expression for this coefficient that led to a limit of trace formulas (e.g., \cite{venk,alt3}), but this expression is valid only for tempered $\pi$. Thus the contribution of nontempered representations, among which the trivial representation is the most nontempered one, has to be removed before such a limit could be analyzed. This led to the perspetive of \cite{FLN}.

As such, the cancellation of the trivial representation only represents a first, though extremely formidable, step in this long process of refining the trace formula.  Nonetheless, as with \cite{alt2}, one can obtain bounds towards Ramanujan as a result of this cancellation, so there are tangible rewards to be gained along the way. 

It is also worth noting that even though the removal of the nontempered contribution can be sidestepped as in \cite{witf}, it still remains to prove the meromorphic continuation of the modified trace formula at hand. This is ultimately what is needed for this refinement of the trace formula, which Arthur calls the $r$-trace formula \cite{problems}. The reason we consider the larger space of test functions $\C(G^Z)$, which are not necessarily compactly supported, is because it contains the {\em basic functions} $f^r_s$ which  the $r$-trace formula requires as inputs \cite{FLGL2,witf}. (This noncompactness alternatively manifests in the limit of trace formulas mentioned above.) 
Then the $r$-trace formula amounts to showing that $J^G(f^r_s)$ has meromorphic continuation in $s\in\CC$. 

\subsection{Poisson apr\`es Matz}

Our second result is inspired by \cite{matz0}, who studied the case of $f^r_s$ with $r$ equal to the standard representation for GL(2). The trace map on $G$ has fibers isomorphic to a space of binary quadratic forms, but here the additive Poisson summation is applied on the level of functions rather than to orbital integrals and thus avoids the analytic difficulties of the first method. Referring to \S\ref{matz} for notation, we have the following variation of Theorem \ref{main}.

\begin{thm}
\label{main2}
Let $f\in \C(G(\A)/C(\A))$.  Then the trace formula for $G$ is equal to the sum of
\[
J^G(f) =  \tr(1(f)) + \sum_{\gamma_{q,X}} \vol(C(\A)G_{\gamma_{q,X}}(F)\backslash G_{\gamma_{q,X}} (\A))\int_{G_{\gamma_{q,X}}(\A)\bs G(\A)} \hat f(x^{-1}\gamma_{q,X} x)  dx,
\]
where $\gamma_{q,X}$ runs over elliptic conjugacy classes of $G(F)$, and the integral over $C(\A)G(F)\backslash G(\A)$ of 
\begin{align*}
&\sum_{X \in  V_G^\mathrm{nell}(F)}\sum_{q\in F^\times} \hat f(x(q,X))  \\
&- \sum_{\xi\in P(F)\bs G(F)} \int_{N(\A)} \sum_{\alpha\in F^\times}f(x^{-1}\xi^{-1}\begin{pmatrix}\alpha&0\\0&1\end{pmatrix} n \xi x) \chi_c(H(\xi x)) dn.\notag
\end{align*}
\end{thm}

\noindent 
As in the first case, the Fourier transform and Poisson formula are taken with respect to the trace variable.
The second and third terms above correspond to \eqref{ell} and \eqref{nell}, the latter being the contribution of regular hyperbolic and unipotent conjugacy classes, regularized by a naive truncation operator. This parallels the additional terms from $\AA_\textnormal{nell}(F)$ in Theorem \ref{main}, but their occurrence here is entirely natural---there is no choice that needs to be made in the completion of the sum. 

\subsection{Application to $L$-functions} For applications, it is also desirable to construct test functions that capture the behaviour of archimedean test functions. We do this in Section \ref{archbas} by constructing the archimedean  basic functions. This allows one to weight the spectral side of the trace formula, for example, with $L$-functions of Maass forms. We then have the following implication.

 \begin{cor}
 \label{corr}
Let $\pi$ be a cuspidal automorphic representation of $G$ unramified away from $S$. 
If $J^G(f^r_s)$ has meromorphic continuation to $s=1$ and is holomorphic in $\mathrm{Re}(s)>1$, then 
the unramified automorphic $L$-function $L^S(s,\pi,r)$ is holomorphic in $\mathrm{Re}(s)>1$ and has meromorphic continuation to $\mathbb C$. In particular, the generalized Ramanujan conjecture holds for $\pi$.
\end{cor}

\noindent Indeed, it was already known  in the works of Langlands, Venkatesh, and Altu\u{g} above that the $r$-trace formula is intimately related to the Ramanujan conjecture. We are simply formulating the connection precisely in this setting.

\section{Poisson apr\`es Langlands}
\label{psf}

\subsection{Measures}\label{meas}

Let us first work with over a local field $F_v$ of $F$.  Fix an invariant differential form $\omega_G$ that is a a generator of the top exterior power of the cotangent bundle on $G$. Following \cite{FLN,gornus}, there are two methods of decomposing the measureon $G$. First, for any maximal torus $T$ of $G$ defined over $F_v$, we also fix an invariant differential form $\omega_T$ of $T$ defined by the characters of $T$. This induces a {\em quotient measure} $d\omega_{T\bs G}$ corresponding to a differential form $\omega_{T\bs G}$ such that $\omega_T\wedge \omega_{T\bs G} = \omega_G$. A standard choice of the the quotient measure when $G$ is unramified over a n $F_v$ nonarchimedean is derived from the canonical measure, which corresponds to a certain measure on a canonical compact subgroup as in \cite{gross}. Closely related to this is the Tamagawa measure on the global product, as both involve normalizations by Artin $L$-functions.

Second, we also have a differential form $\omega_{c^{-1}(a)}$ defined on the fibres of the map $c: G\to \AA$, again characterized by $\omega_\AA\wedge \omega_{c^{-1}(a)} = \omega_G$. It induces a measure on the stable orbit $c^{-1}(c(\gamma))$ for any regular $\gamma \in G(F_v)$, which we call the {\em geometric} measure following \cite{gor}. In general, $c^{-1}(c(\gamma))$ is a union of rational orbits of $\gamma$ in $G(F_v)$, but in our case it consists of a single rational orbit. The measure on $\AA$ is, up to a set of measure zero, given as a disjoint union of measures of stable $F_v$-conjugacy classes of maximal $F_v$-tori in $G(F_v)$, so long as we divide by the order of the Weyl group of each torus and multiply by the Jacobian. We refer to \cite{gornus} for a careful treatment of these notions. 

In any case, we have that for any $f\in \C(G(F_v))$, 
\be\label{Acaloc}
\int_{G(F_v)}f(g) d\omega_G  = \int_{\AA(F_v)} \int_{c^{-1}(a)} f(g) d\omega_{c^{-1}(a)}\ d\omega_\AA,
\ee
which can be compared with the Weyl integration formula in \cite[\S4]{gor} (see also \cite[(24)]{gornus}). Note that in loc. cit. the authors require $f$ to be compactly supported, but the identity holds more generally so long as $f$ is integrable over $G$. 

For our computations in Theorem \ref{main}, we shall work with the normalization of measures used in \cite[\S3.2]{FLN}, including the geometric measure, and the relation to the quotient or canonical measure is explicated in \cite{gornus}. For the purposes of additive Poisson summation, we also fix an additive character $\psi = \prod \psi_v$ of $\A$, which induces a unique measure on each $F_v$ such that the Fourier transform with respect to $\psi_v$ is self-dual \cite{tate}. 

\subsection{Trivial representation}

With this it is a simple matter to compute the trivial representation globally.
\begin{lem}\label{Aca}
 The trivial character can be expressed as
\[
\tr(1(f)) = \int_{G(\A)} f(g)d\omega_G =  \int_{\AA(\A)}\int_{\c^{-1}(a)} f(g) d_\mathrm{geo}g\ da,
\]
where we denote by $d_\text{geo}g$ and $da$ respectively the product of local measures above.
\end{lem}

\begin{proof}
By abuse of notation, we also write $\omega_G$ for the global measure taken as the adelic product of local geometric measures as in \eqref{Acaloc}. This is also computed in \cite[\S5]{FLN}.
\end{proof}

\subsection{The regular elliptic contribution}

Turning to the geometric side of the trace formula, denote the integral with respect to the local geometric measure 
\[
\theta_{f_v}(a)=\int_{c^{-1}(a)} f_v(g) d\omega_{c^{-1}(a)}.
\]
Also, let $D(\gamma)$ be the usual Weyl discriminant of $\gamma$. We have the following relation between measures. 

\begin{lem}\label{lem2}
For any $f_v\in \C(G(F_v)^Z)$, we have
\be
\label{thetaf}
\theta_{f_v}(a) = |D(\gamma)|_v^{\frac12}\int_{T\bs G} f_v(\mathrm{Ad}(g^{-1}) \gamma) d\omega_{T\bs G}.
\ee
If $f\in \H(G,K)$, the function $\theta_f$ is compactly supported and smooth on $\AA(F_v)$ except at two points where it is continuous but nonsmooth.
\end{lem}

\begin{proof}
This identity \eqref{thetaf} relates the geometric measure to the usual quotient measure. It is computed in \cite[(3.31)]{FLN} and \cite[Theorem 3.11]{gornus} for smooth compactly supported test functions, and we note that the computation holds for $\C(G(F_v)^Z)$ since the identity is true so long as the integrals are convergent. Thus $f_v$ determines a function $\theta_f$ on $\AA(F_v)$ that is given by the normalized orbital integral of $f_v$ along the orbit of $\gamma$ in $G(F_v)$.  The orbital integral of a compactly supported test function is again compactly supported, hence $\theta_{f_v}$ is compactly supported on $\AA(F_v)$.

For the singularities of orbital integrals on $G=\GL(2)$, it suffices to consider instead the derived group $G_\text{der} = \mathrm{SL}(2)$. To begin with, it is a fundamental result of Harish-Chandra that the normalized orbital integral 
\[
|D(\gamma)|_v^{\frac12}O(\gamma,f_v)
\]
is bounded and, for nonarchimedean $v$, locally constant for regular semisimple $\gamma$.
Restricting to $f\in \H(G,K)$, the desired properties follow by Langlands' analysis in \cite[\S3.1--3.2]{ST} for real, complex, and nonarchimedean $v$. For archimedean $v$, Langlands shows that $\theta_{f_v}(a)$ is continuous as $a$ approaches the singular set.

In the case of $F_v$ is nonarchimedean, Langlands avoids residual characteristic 2 to make explicit the Shalika germ expansion of $O(\gamma,f_v)$, but as Langlands notes this is well known in general characteristic. For example in \cite[\S5]{harm} Kottwitz explicitly computes for any nonarchimedean local field the expansion at regular elements $\gamma\in \GL_2(\mathcal O_v)$, where $\mathcal O_v$ is the ring of integers of $F_v$, as
\[
O(\gamma,f) =  A_1(\gamma)\mu_1(f) + A_2(\gamma)\mu_2(f).
\]
Here $A_1(\gamma),A_2(\gamma)$ are uniquely determined complex numbers and $\mu_1(f),\mu_2(f)$ are the orbital integrals of $f$ at the unipotent elements
\[
\begin{pmatrix}1&0\\0&1\end{pmatrix},\quad \begin{pmatrix}1&1\\0&1\end{pmatrix}
\]
respectively. Then \cite[\S4.2]{gornus} shows that $\theta_{f_v}$ is continuous at the singularities.
\end{proof}

\subsection{The problem of completion}
\label{comp}

In order to proceed, several remarks are in order. First, in the setup of \cite[\S6]{ST}, Langlands applies several approximations to the analogous sum in  \eqref{aell}: the typical global quotient measure (called the {\em canonical} measure) chosen is scaled by an Artin $L$-value of the algebraic torus $T$ determined by the elliptic conjugacy class. The relation between measures is carefully explicated in \cite{gornus}, and implicitly uses the fact that the Tamagawa number of $T$ is 1 \cite[Remark 5.4]{gornus}. This $L$-value diverges if one evaluates the infinite product at $s=1$ for split tori. Langlands thus begins with a finite product  at Re$(s)>1$ and approximates the $L$-value with 1. But in considering the split terms, Langlands retains this $L$-factor, which precisely diverges when the associated torus splits. This is the point at which Langlands' analysis in \cite{ST} essentially concludes. 

This obstruction appears again for the same reason in \cite{alt1,melk}. Here, the problem is overcome by expressing the $L$-value in terms of a truncated Dirichlet series indexed by an appropriate quadratic symbol (this is the approximate functional equation, a standard tool in analytic number theory). Because the symbol itself is carefully constructed so as to be well-defined for all $a\in \AA(F)$, the resulting sum can be extended to the entire base $\AA(F)$. After these manipulations, the Poisson summation formula was then able to be applied in \cite{alt1,melk} to the image of $\tr(\gamma)$ in $\AA(F)$ accordingly. 

In our case, we are careful to use the geometric measure in the sense of \eqref{thetaf}, rather than the canonical measure, to define the summands and complete the sum, and is the key insight needed for Theorem \ref{main}. It should be understood as an alternate way of extending the sum such that the additional terms do not cause it to diverge. 
We thus complete the summation above by adding and subtracting the missing terms to obtain a sum over a complete lattice. (See also Remark \ref{comp2} below.)

\begin{lem} \label{lem3}
Let $f\in C_c^\infty(G^Z)$. The regular elliptic contribution \eqref{aell} is equal to
\be
\label{aall}
J^G_\mathrm{ell}(f) = \sum_{a\in \AA(F)}\theta_f(a) - \sum_{a\in \AA_\mathrm{spl}(F)}\theta_f(a) - \sum_{a\in \AA_\mathrm{sing}(F)}\theta_f(a),
\ee
where $\theta_f(a) = \prod_{v}\theta_{f_v}(a)$ is a conditionally convergent product.
\end{lem}
\begin{proof}
The typical expression for the regular elliptic part of the trace formula
\[
J^G_\mathrm{ell}(f) =  \sum_{\gamma}\text{vol}(G_\gamma(F)\backslash G_\gamma(\A)^Z) \int_{G_\gamma(\A)\backslash G(\A)} f(x^{-1} \gamma x) dx,
\]
is taken with respect to the canonical measure (c.f. \cite[\S4]{gornus}). The local orbital integral with respect to the canonical measure is simply a rewriting of 
\[
|D(\gamma)|_v^{\frac12}\int_{G_\gamma(F_v)\backslash G(F_v)} f(x^{-1}_v \gamma x_v) dx_v =  |D(\gamma)|_v^{\frac12}\int_{T\bs G} f_v(\mathrm{Ad}(g^{-1}) \gamma) d\omega_{T\bs G}
\]
as in Lemma \ref{thetaf} above. Applying the lemma at each place, as in the measure conversion of \cite[Theorem 5.3]{gornus}, it can be re-expressed in terms of the geometric measure as 
\be
\label{aell}
J^G_\text{ell}(f) = \sum_{a\in \AA_\text{ell}(F)}\theta_f(a), \qquad f\in \C(G^Z)
\ee
which is simply a rewriting of the regular elliptic part of the trace formula (see also \cite[\S4]{FLN}). Note that as usual we have that the product $\prod_v |D(\gamma)|_v=1$, and we supress the factor of $2\pi$ that appears in \cite[Theorem 5.3]{gornus}, which we can by rescaling the archimedean contribution. 

To extend the sum, we use the property that $\theta_f$ is smooth over $\AA_\text{spl}(F)$ and well-defined over $\AA_\text{sing}(F)$ by \cite[\S3]{ST}. This allows us to add and subtract the missing terms from the sum, arriving at the desired expression.
\end{proof} 

\subsection{Convergence of the dual sum}

For the convenience of the reader, we recall Langlands' proof of the convergence of the dual sum. We first consider functions $h_v$ on a local field $F_v$. As usual we fix an additive character $\chi_v$ on $F_v$ and a measure on $F_v$ such that the Fourier transform with respect to $\chi_v$ is self-dual. If we replace $\chi_v$ with $x \mapsto \chi_v(ax)$ for some $a\in F_v^\times$, the Fourier transform $\hat h_v(x)$ becomes $|a|^{1/2}\hat h_v(ax)$. Langlands then obtains the following growth estimates.

\begin{lem}\cite[Lemme 5.1]{ST}
Let $h_v(x)$ be a function on $\R$ such that for $x\neq0$ it is equal to the product of $|x|^{\lambda -1},\lambda>1$ and a compactly supported $\phi$ that is smooth for $x\neq0$ and such that the limit
\[
\lim_{x\to\pm0} \phi^{(n)}(x), \quad n\ge0
\]
and its derivatives exist. Then the Fourier transform 
\[
\hat h(x) = \int_\R h(x) e^{ixy} dx
\]
is $O(|y|)^{-\lambda}$ as $y\to\infty$.
\end{lem}

\begin{lem}\cite[Lemme 5.2]{ST}
Let $h_v(z)$ be a function on $\mathbb C$ that is smooth and compactly supported for $z\neq 0$, and in a neighborhood of $z=0$ it is equal to the product of $|z|^{\lambda -1},\lambda>1$ and a smooth function. Then the Fourier transform 
\[
\hat h_v(w) = \int_\mathbb C h_v(z) e^{i\mathrm{Re}(zw)} dz
\]
is $O(|w|)^{-\lambda}$ as $|w|\to\infty$.
\end{lem}

\begin{lem}\cite[Lemme 5.3]{ST}
Let $h_v(x)$ be a function of $F_v$ nonarchimedean given by the product of $|x|^{\lambda -1},\lambda>1$ and a smooth function $\phi$ whose support is the intersection of a compact open of $F_v$ and a class in $F_v^\times/(F^\times)^2$. Then the Fourier transform 
\[
\hat h_v(x) = \int_{F_v} h_v(x) \chi(xy) dx
\]
is $O(|y|)^{-\lambda}$ as $y\to\infty$.
\end{lem}

The purpose of the previous results if the following convergence condition.

\begin{lem}\cite[Lemme 5.4]{ST}\label{dual}
Let $S$ be a finite set of valuations of $F$ and $h = \prod_{v\in S} h_v$. If there exists positive constants $c,d$ such that $|h_v(a)|\le c \min(1,|a|^{-1-d}_v)$, for all $a\in F_v$ and $v\in S$, then
\[
\sum_{a\in F_S}|h(a)|<\infty.
\] 
\end{lem}

\noindent This will provide the convergence of the dual sum in the Poisson formula below.

\subsection{Poisson summation}

We first recall the adelic Poisson summation formula as originally established by Tate.

\begin{lem}\cite[Lemma 4.2.4]{tate}
\label{PSF}
Let $\varphi$ be a function on $\A$ such that 
\ben
\item
$\varphi$ is continuous and belongs to $L^1(\mathbb A)$, 
\item
for all $y\in F\backslash \mathbb A$,
\[
\sum_{x \in F}\varphi(x+y)
\]
is uniformly convergent, and 
\item the dual sum converges absolutely
\[
\sum_{x\in F}|\hat \varphi(x)|<\infty.
\]
\een
Then 
\[
\sum_{x \in F }\varphi(x) = \sum_{y\in F} \hat\varphi(y).
\]
\end{lem}

\noindent In particular, there is no requirement for $\varphi$ to be smooth or even differentiable.

With these preparations, we thus arrive at the following.

\begin{prop}
Let $f\in C_c^\infty(G^Z)$. Then the sum \eqref{aall} is well-defined, and is equal to
\be\label{theta}
\sum_{a\in \AA(F)}\hat\theta_f(a) - \sum_{a\in \AA_\mathrm{spl}(F)}\theta_f(a) - \sum_{a\in \AA_\mathrm{sing}(F)}\theta_f(a)
\ee
\end{prop}

\begin{proof}
The proof follows from the preceding discussion and verifying the conditions of Lemma \ref{PSF} to the first term. The first is given by Lemma \ref{lem2}, the second follows from the compact support of $\theta_f$, and the third is given by Lemma \ref{dual}.
\end{proof}

\noindent The main Theorem \ref{main} for $f\in C_c^\infty(G^Z)$ then follows from the above formula and Lemmas \ref{Aca} and \ref{lem3}.

\subsection{Completion of main theorem}
\label{cont}

We want to extend Theorem \ref{main} to $f\in\C(G^Z)$. Let us denote the distribution defined by \eqref{theta} as $f\to \theta(f)$. Finis and Lapid \cite[Theorem 1]{FLGL2} show that the original expression \eqref{aell} holds for $f\in\C(G^Z)$ by establishing the existence of a continuous seminorm on $\C(G^Z)$ that extends $C_c^\infty(G^Z)$. To see the same for the dual sum, we shall show that the same holds. 

\begin{prop}
$\theta(|f|)$ is a continuous seminorm on $\C(G^Z)$.
\end{prop}

\begin{proof}
The statement for the second and third terms in \eqref{theta} follow quickly from a simpler version of \cite[p.389]{FLGL2} which bounds the actual hyperbolic contribution to the trace formula. The main term to consider is the first sum
\[
\sum_{a\in \AA(F)}\hat\theta_f(a) = \sum_{a\in\AA(F)}\int_{\AA(\A)} \psi(ab) \int_{c^{-1}(b)} |f(g)| d_\text{geo}g \ db,
\]
where $\psi$ is a continuous additive character on $\AA(\A)$ that is self-dual with respect to the measure on $\AA(\A)$. We may take for example the product of local characters $\psi_v(a_v) = e^{-2 \pi i a_v}$. 
As in \eqref{Acaloc}, by the Weyl integration formula this is equal to 
\be
\label{psia}
\sum_{a\in\AA(F)} \int_{G(\A)^Z} \psi(a \cdot \tr(g))|f(g)|dg,
\ee
where we note that we can view $\tr(g)$ here as the trace for the standard representation of $G$.

Let $\A_f = \prod_{v<\infty} F_v$ be the product of finite adeles, and let $K_f = \prod_{v<\infty} K_v$, where $K_v$ is a maximal compact subgroup of $G(F_v)$. Since we are working adelically, we can assume now for simplicity that $F=\Q$. Let $T_{r,N}$ be the characteristic function of the double coset 
\[
K_f \begin{pmatrix} r & 0 \\ 0 & r \end{pmatrix} \begin{pmatrix} N & 0 \\ 0 & 1 \end{pmatrix}  K_f,
\]
which forms a set of representatives of $K_f \bs G(\A_f) / K_f$ for all positive $r\in \Q$ and $N\ge1$.  Following \cite[\S4]{FL}, it suffices to show that there exists a continuous seminorm $\mu$ on $\C(G(\R)^1)$ such that the expression $\theta(|f|)$  for 
\[
f(g) = f_\infty(g_\infty)T_{r,N}(g_f),\qquad  g_\infty \in G(\R)^1, g_f \in G(\A_f), f_\infty\in \C(G(\R)^1)
\]
is bounded by $\mu(f_\infty)$ times a constant depending at most on $r,N$. Using the fact that  
\[
\int_{G(\A_f)}T_{r,N}(g)dg = \text{vol}(K_f \begin{pmatrix} r & 0 \\ 0 & r \end{pmatrix} \begin{pmatrix} N & 0 \\ 0 & 1 \end{pmatrix}  K_f)=N \prod_{p|N}\left(1+\frac1p\right)
\]
we see that for $f$ as above, the sum \eqref{psia} is equal to $O( N ||f_\infty||_1)$.
\end{proof}

\section{Poisson apr\`es Matz}
\label{matz}

\subsection{Parametrizing the elliptic terms}

We now turn to the second method, inspired by the work of Matz. Recall that the kernel of the integral operator used to define the trace formula can be written as
\[
\int_{C(\A)G(F)\backslash G(\A)} K(x,x) dx = \int_{C(\A)G(F)\backslash G(\A)}\sum_{\gamma \in G(F)} f(x^{-1}\gamma x) dx, 
\]
where $C$ is the center of $G$. (We switch conventions here and consider $\bar G(\A) = G(\A)/C(\A)$ in place of $G(\A)^Z$ above. See \cite[\S6]{knapp} for passing between the two.) In general, the integral does not converge when $G$ is not anisotropic, in which case it is necessary to regularize the integral. In our case, it is enough to use the naive truncation as in \cite[(6.4)]{GJ} for some parameter $c>0$ tending to infinity,
\be
\label{kernel}
J^G(f)=\tr(R(f)) = \int_{C(\A)G(F)\backslash G(\A)} \Lambda^c_2 K(x,x) dx,
\ee
where the truncated kernel $\Lambda^c_2 K(x,x)$ is equal to
\be
\label{kernel2}
\sum_{\gamma \in G(F)} f(x^{-1}\gamma x) - \sum_{\xi\in P(F)\bs G(F)} \int_{N(\A)} \sum_{\gamma\in G(F)}f(x^{-1}\xi^{-1}\gamma n \xi x) \chi_c(H(\xi x)) dn
\ee
referring to \cite{GJ} for notation. The lefthand side of \eqref{kernel} as written does not depend on $c$, because in practice either one takes the limit as $c$ tends to infinity or evaluates at the distinguished value $c=1$. 

The main observation here is that we can apply Poisson summation to the first sum on the level of {\em functions} on $G$, as was done by Matz in \cite{matz0}. The test functions used by Matz are of the form
\be
\label{basics}
f_s^\text{std}(g) = \int_{\A^\times}|\det(ag)|^{s +1/2} \Phi(ag) da
\ee
where $\Phi$ is a Schwartz-Bruhat function on the Lie algebra $\mathfrak{g}(\A)$, and $f_s^\text{std} \in \C(\bar G)$ for Re$(s)>\frac32$. This produces the standard $L$-functions on the spectral side of the trace formula for $G$. 

We will transform the main term of $J^G(f)$ in \eqref{kernel}, which is
\be
\label{Kxx}
\int_{C(\A)G(F)\backslash G(\A)}\sum_{\gamma \in G(F)} f(x^{-1}\gamma x) dx,
\ee
as follows. Let $V\simeq \mathbb G_a^3$ be the space  of binary quadratic forms, where the triple $X = (X_1,X_2,X_3)$ represents the form $X(u,v) = X_1u^2 + X_2uv + X_3v^2$. There is an action of $G$ on $V$ given by $X(u,v)\cdot x = X((u,v)x^t)$ and also of $\GL(1)$ given by scaling.  There is a natural isomorphism $\mathfrak{gl}_2 = \text{Mat}_2\simeq V\oplus \mathbb G_a$ where the map on the second component is given once again by the trace \cite[V.i]{matz0}. The adjoint action of $G$ on $\g$ splits into the action of the group $\{(\det x^{-1},x): x\in G\}$ on $V$ and the identity on $\mathbb G_a$. Let $V_G \subset V$ be the image of $G$ in $V$.

We thus reindex the summation as 
\[
\int_{C(\A)G(F)\backslash G(\A)}\sum_{X \in V_G(F)}\sum_{q\in F} f(x(q,X) ) dx,
\]
where $x (q,X)$ denotes the action of $x$ on the second component by $(\det x^{-1},x)$.
Fix a nontrivial additive character of $\A$. The Fourier transform with respect to $q$ is
\be
\label{fourier}
\hat f(x(q,X)) = \int_\A  f(x(q,X)) \psi(aq)da
\ee
for fixed $x$ and $X$. 

\subsection{Poisson summation}

The key to our second main Theorem \ref{main2} is the following.

\begin{prop}
Let $f\in \C(\bar G)$. Then $J^G(f)$ is equal to the sum of 
\[
 \tr(1(f)) + \int_{C(\A)G(F)\backslash G(\A)}\sum_{X \in V_G(F)}\sum_{q\in F^\times} \hat f(x(q,X))  dx,
\]
where $\hat f$ denotes the additive Fourier transform of $f$ as defined in \eqref{fourier}, and
\[
- \int_{C(\A)G(F)\bs G(\A)} \sum_{\xi\in P(F)\bs G(F)} \int_{N(\A)} \sum_{\gamma\in G(F)}f(x^{-1}\xi^{-1}\gamma n \xi x) \chi_c(H(\xi x)) dn\ dx.
\]
\end{prop}

\begin{proof}
The function $f\in \C(\bar G)$ descends to one on $\C(\A)$, then Poisson summation over $q$ as in \cite{tate} (c.f. the proof of Lemma \ref{tate}) then gives
\[
\int_{C(\A)G(F)\backslash G(\A)}\sum_{X \in V_G(F)}\sum_{q\in F} \hat f(x(a,X))  dx,
\]
and the zeroth term is 
\[
\int_{C(\A)G(F)\backslash G(\A)}\sum_{X \in V_G(F)} \int_\A  f(x(a,X)) da\ dx.
\]
The space of binary quadratic forms $V$ is a prehomogeneous vector space under the action of $G$. The outer integral and sum together range over all $V_G(\A)$, up to $C(\A)$, we therefore obtain the trace of the trivial representation,
\[
\tr(1(f))=\int_{C(\A)\backslash G(\A)} f(x) dx
\]
as claimed.

We note that the main term \eqref{Kxx} in $J^G(f)$ does not converge on its own, and applying the truncation by subtracting the second term in \eqref{kernel2} allows the expression to converge.
\end{proof}

\begin{rem}\label{comp2}
The intersection of $V$ with the regular elliptic locus of $G(F)$ is a subset $V_G^\text{ell}$ given by elements whose discriminant is not a square in $F$. In comparison, where Matz considers the terms in the trace formula {\em after} truncation, and applies the summation formula as above to the regular elliptic contribution only. This amounts to restricting the sum over $V_G(F)$ above to $V_G^\text{ell}(F)$. Matz shows that the analytic behaviour of the latter is described by Shintani zeta functions for definite and indefinite binary quadratic forms \cite{matz2} (compare also Corollary 49 and Proposition 57 of \cite{matz0}). Our extension of the sum to $V_G(F)$ could be interpreted as a `completion' that allows the character of the trivial representation to appear in the dual sum, in the sense of \S\ref{comp}.
\end{rem}

\begin{proof}[Proof of Theorem $\ref{main2}$]
We recall that the truncation operator acts trivially on the elliptic contribution, so if we isolate the term 
\[
\int_{C(\A)G(F)\backslash G(\A)}\sum_{X \in V_G^\text{ell}(F)}\sum_{q\in F} \hat f(x(q,X))  dx,
\]
we can write this again as an elliptic orbital integral as follows. Let $\gamma_{q,X}$ be the elliptic conjugacy class in $G(F)$ corresponding to the pair $(q,X)$, then the latter is equal to
\be
\label{ell}
\sum_{\gamma_{q,X}} \vol(C(\A)G_{\gamma_{q,X}}(F)\backslash G_{\gamma_{q,X}} (\A))\int_{G_{\gamma_{q,X}}(\A)\bs G(\A)} \hat f(x^{-1}\gamma_{q,X} x)  dx.
\ee
Note that this includes the singular elliptic contribution.

The remaining terms, which are equal to the complement $V_G^\text{nell}(F) = V_G(F) - V_G^\text{ell}(F)$ and with $q\in F^\times$, correspond to regular hyperbolic and unipotent elements in $G(F)$ with nonzero trace. The remaining contribution can then be written as the integral over $C(\A)G(F)\backslash G(\A)$ of 
\begin{align}\label{nell}
&\sum_{X \in V_G^\text{nell}(F)}\sum_{q\in F^\times} \hat f(x(q,X))  \\
&- \sum_{\xi\in P(F)\bs G(F)} \int_{N(\A)} \sum_{\alpha\in F^\times}f(x^{-1}\xi^{-1}\begin{pmatrix}\alpha&0\\0&1\end{pmatrix} n \xi x) \chi_c(H(\xi x)) dn,\notag
\end{align}
where the second term is rewritten as in \cite[(6.8)]{GJ}. 
\end{proof}

As with the original trace formula, one would like to transform the remaining terms \eqref{nell} into weighted orbital integrals. The trace of a diagonal element is invariant under multiplication by $n$. Let $X_{\alpha,n} \in V(\A)$ be the element corresponding to $(\begin{smallmatrix}\alpha&0\\0&1\end{smallmatrix})n$. The above integrand is then equal to
\[
 \sum_{\alpha\in F^\times}f(x \xi (\alpha+1, X_{\alpha,n})) \chi_c(H(\xi x)).
 \]
To apply Poisson, we have to add and subtract the term corresponding to $\alpha=0$, but the element $(\begin{smallmatrix}0&0\\0&1\end{smallmatrix})n$ does not lie $G$.    If we  assume additionally that  $f$ is  a function that pulls back to $\g$, as with \eqref{basics} and Lemma \ref{Gg} below, then the missing term is well-defined, giving 
 \[
 \sum_{\alpha\in F}\hat f(x \xi (\alpha+1, X_{\alpha,n})) \chi_c(H(\xi x)).
 \]
Making a change of variables $q = \alpha +1$ does not change the value of the integral, and doing so we arrive at the integral over $C(\A)G(F)\backslash G(\A)$ of 
\begin{align*}
&\sum_{X \in V_G^\text{nell}(F)}\sum_{q\in F^\times} \hat f(x(q,X))  + (q = 1) \\ 
&- \sum_{\xi\in P(F)\bs G(F)} \int_{N(\A)} \sum_{q\in F^\times}\hat f(x \xi (q, X_{q-1,n})) \chi_c(H(\xi x)) dn,
\end{align*}
where $(q = 1)$ denotes the correction term added. The regular unipotent terms correspond to $q=2$, and the rest are regular hyperbolic. We do not continue the analysis here, but note that in treating the regular hyperbolic terms \cite[V.i]{matz0} Matz must also add corrections terms whose entries have determinant 0 in order to apply Poisson summation  on the trace variable.

\section{Archimedean basic functions}

\label{archbas}

\subsection{A Fr\'echet algebra}

In this section, we abuse notation and set $G(\A)^1 = \{ g \in G(\A) : |\det(g)| = 1\}$, and denote $\C^1(G) = \C(G(\A)^1)$. Also let $\C^2(G(\A)^1;K)$ be the space of smooth right $K$-invariant functions on $G(\A)^1$ which belong to $L^2(G(\A)^1)$ along with all its derivatives. It is again a Fr\'echet space under the family of seminorms $|| f * X||_{2}$ for $f\in \C^2(G(\A)^1;K)$. We note that this differs slightly from the Schwartz algebra considered in \cite{isolation}, in that we do not require $f_v$ to be compactly supported for all nonarchimedean $v$. Denoting the spaces $\C^2(G)$ and $C_c^\infty(G^1)$ analogously, we note that $C_c^\infty(G^1)$ is also dense in $\C^2(G)$, hence the intersection $\C^1(G)\cap \C^2(G)$ is dense in both spaces. Furthermore, we can restrict our attention to a subspace of this intersection, which we next describe.

Let $F_v$ be a nonarchimedean, and denote by $G_v = G(F_v)$ and $K_v \subset G_v$ a maximal compact subgroup. Let $\H(G_v;K_v)$ be the usual $K_v$-spherical Hecke algebra. Given the canonical homomorphism $H_G:G(F_v) \to \a_G$, where $\a_G = \text{Hom}_\Z(X^*(G)_F,\R)$, let $\a_{G,v}$ be the image of $H=H_G$ in $\a_G$. Then define the almost compactly supported Hecke algebra $\H_\text{ac}(G_v;K_v)$ to be the set of functions $f:G_v\to \CC$ such that for all $b\in C_c(\a_{G,v})$, the product $b(H_G(\cdot))f(\cdot)$ belongs to $\H(G_v;K_v)$. The Satake isomorphism extends to $\H_\text{ac}(G_v;K_v)$. Let $S$ be a finite set of places of $F$ containing the archimedean valuations and outside of which $G$ is unramified. Define
\[
\C^0(G)= \{ f = \otimes_v f_v \in \C^1(G)\cap \C^2(G): \forall v\not\in S, f_v \in \H_\text{ac}(G_v;K_v)\}.
\]
This is the space that we shall consider. It contains the generalised Schwartz space of \cite{BK}, and in particular the basic functions that we need.

We first show that the method of \cite{isolation} isolating the cuspidal spectrum applies to the space $\C^0(G)$. We state it only in the restricted form that we shall need. Recall that a multiplier of $\C^0(G)$ is a complex linear operator $\C^0(G)\to \C^0(G)$ that commutes with left and right multiplication.

\begin{lem}
\label{isolation}
Let $\pi$ be a cuspidal automorphic representation of $G(\A)$ with trivial central character. Then there exists a multiplier $\mu$   such that for every $f\in \C^0(G)$, the right regular representation $R(\mu*f)$ maps $L^2(G(F)\backslash G(\A)^1/K)$ into its $\pi$-isotypic subspace, and $\pi(\mu*f) = \pi(f)$.
\end{lem}

\begin{proof}
We refer to \cite{isolation} for precise notation. Since the archimedean component $f_\infty$ of $f$ belongs in the Schwartz algebra $\C^2(G)$, 
the archimedean multiplier $\mu_\infty$ \cite[Theorem 3.15]{isolation} remains valid. On the other hand, to see that the full multiplier $\mu_{(M,\sigma)} = \mu^0_\infty \cdot \mu^\dagger \cdot \nu^\dagger$ of  \cite[Theorem 3.19]{isolation} extends, it suffices to observe that the nonarchimedean multiplier $\nu^\dagger$ belonging to the product of spherical Hecke algebras over almost all unramified places extends to an endomorphism of $\mathcal H_\text{ac}(G_v,K_v)$ under the convolution product.
\end{proof}

This allows one to isolate the cuspidal terms on the spectral side of the trace formula weighted by automorphic $L$-functions, but it does not give information about the geometric side. For that, we have to consider  basic functions more explicitly. 

\subsection{Basic functions} 

We are  interested in the global test function
\begin{equation}
\label{basic}
F^r_s = f_S \times \prod_{v\not\in S} b^r_{s,v}, \qquad s\in\CC
\end{equation}
where $f_S \in \C^1(G(F_S))$ and $b^r_{s,v}$ is the basic function on $G(F_v)$ characterised by tr$(\pi(b^r_{s,v})) = L_v(s,\pi,r)$ for any unramified smooth admissible representation $\pi_v$ of $G(F_v)$ with $F_v$ nonarchimedean. Also, set
\begin{equation}
\label{basic2}
f^r_s(g)= \int_{\A^\times} F^r(ag)|\det(ag)|^{s+\frac12} da. 
\end{equation}
It converges absolutely for Re$(s)>0$ (c.f. Lemma \ref{tate}) and  furthermore it is known that $f^r_s\in \C(G(
\A)^1)$ for Re$(s)$ large enough. The trace formula applied to \eqref{basic}, denoted by $J^G(f^r_s)$, is a distribution on $\C^1(G(F_S))$, but for simplicity we may also restrict to $f_S \in C_c^\infty(G(F_S)^1)$.

The existence of $b_{s,v}^r$ is well-known for nonarchimedean $F_v$, where it is based on an application of the Satake isomorphism \cite[\S5.7]{BK}. The following lemma records the archimedean analogue, which  is related to classsical Meijer G-functions (see for example \cite{G1,G2}). This will allow us to consider the $L$-function at infinite places also.

\begin{lem}
\label{lembas}
Let $G$ be an unramified quasisplit reductive group over $F_v$ archimedean and let $r:\hat G(\CC) \to \GL(V)$ be a finite-dimensional complex representation. Then there exists $b_{v,s}^r \in \C^1(G(F_v),K_v)$ such that $\tr(\pi_v(b_{s,v}^r))= L_v(s,\pi,r)$ for $\mathrm{Re}(s)\gg 0$ and any irreducible admissible representation $\pi_v$  of $G(F_v)$.
\end{lem}

\begin{proof}
By the archimedean local Langlands correspondence \cite{LLCR}, $L_v(s,\pi,r)$ can be identified with an Artin $L$-factor $L(s,r\circ \phi_{\pi_v})$, where $\phi_{\pi_v}$ represents a $\hat G(\CC)$-conjugacy class of homomorphisms from from the Weil group of $F_v$ to $^LG$ associated to $\pi_v$. Hence it is given by a finite product of gamma functions whose argument is a linear combination of $s$ and the components of $\mu(\pi_v)$, the spherical conjugacy class in $\hat G(\CC)$ corresponding to $\pi_v$. In particular, $L(s,r\circ \phi_{\pi_v})$ is a finite product of Mellin transforms of exponential functions. We can thus view $L_v(s,\pi,r)$ as an element $\phi(\mu(\pi_v))$ of $\CC[[\mathfrak h^*]]^W$, where $\mathfrak h^*_\CC = X^*(T)\otimes_\Z\CC$ is the dual of a Cartan subalgebra of $\g\otimes_\R \CC$.

On the other hand, the Harish-Chandra isomorphism is the archimedean analogue of the Satake isomorphism, identifying the center of the universal enveloping algebra $\mathcal Z(U(\g_\CC))$ with the ring of algebraic functions $\mathbb C[\mathfrak h^*_\mathbb C]^W$, where $W$ is the Weyl group of $G$ relative to $T$ \cite[Part III]{HC1}. This allows us to identify the infinitesimal character $\chi_{\pi_v}$ of $\pi_v$ with a $W$-orbit of $\mathfrak h^*_\CC = X_*(\hat T)\otimes_\Z \CC$. Putting these together, it follows that there exists a bi-$K$-invariant function on $G(F_v)$ satisfying the desired properties.
\end{proof}

For $r = \text{Sym}^k$, it follows from \cite[(2.22)]{monoid} and \cite[\S4]{hahn} that the monoid  associated  to $b_v^r$ can be identified with $\g(F_v)$ for nonarchimedean $v$. The archimedean  $L$-factors are computed explicitly in \cite{MS}.

\begin{lem}
\label{Gg}
$b^r_{s,v}\in \C^1(G(F_v))$ and pulls back to an element of $\C^1(\g(F_v))$.
\end{lem}

\begin{proof}
For $v$ nonarchimedean, $b^r_{s,v}$ is determined explicitly by its values on the cocharacter group $X_*(T)$ where $T$ is an $F_v$-split torus \cite[\S3]{li}, in particular it is defined on $\g(\A)$. Also by \cite[\S3.2]{li} it belongs to $L^2(G(F_v)))$ for Re$(s)>0$ and $L^1(G(F_v))$ for Re($s)$ large enough. The archimedean case follows similarly by the proof of the preceding lemma. 
\end{proof}

We also record the following relation between the analytic continuation of the distribution $J^G(f^r_s)$ and $L$-functions.

\begin{proof}[Proof of Corollary $\ref{corr}$]
Analytic continuation follows from Lemmas \ref{isolation} and \ref{lembas}. This includes symmetric power representations $r=\text{Sym}^k$ for all $k\ge1$, and as it is well-known (e.g. \cite{LRS} and \cite[p.48]{murty}) this would imply the Ramanujan conjecture for $\pi$. 
\end{proof}

\noindent 
 One might hope, perhaps naively, that this concretely translates the problem into a geometric one that can be studied from a new perspective. We treat the regular unipotent contribution as a simple exercise. 

\subsection{Regular unipotent terms}

We conclude with a simple observation about the regular unipotent terms in the trace formula. (Note that the removal of the $\tr(\gamma)\neq 0$ terms does not affect the unipotent terms.) Let $\mathcal O'$ be an open subgroup of $\mathcal O_f = \prod_{v<\infty}\mathcal O_v$, and let $\C(\A;\mathcal O')$ denote the space of smooth $\mathcal O'$-invariant functions $\varphi$ on $\A$ such that $\varphi^{(n)}\in L^1(\A/\mathcal O')$ for all $n\ge0$. 

\begin{lem}
\label{tate}
Let $\varphi\in \C(\A;\mathcal O')$. Then the zeta integral
\[
z(\varphi,s) = \int_\A \varphi(x) |x|^s dx
\]
has meromorphic continuation to $s\in\CC$ with at most simple poles at $s=0,1$ with residues $-\vol(F^\times \backslash \A^\times)\varphi(0)$ and $\vol(F^\times \backslash \A^\times)\hat\varphi(0)$ respectively. 
\end{lem}

\begin{proof}
Recall from Tate's thesis \cite[Theorem 4.4.1]{tate} that this is essentially equivalent to establishing the Poisson summation formula for $\varphi$, with the poles occurring if the test function $\varphi$ satisfies the conditions as in the proof of \cite[Theorem 4.4.1]{tate}. We refer again to Lemma \ref{PSF}. Property (i) is known in general for basic functions by \cite{li}. The uniform bound on the first sum is given in \cite[Lemma 3.4]{FLGL2}. 
 For the dual sum, we first observe that 
\[
(1 + |x|^2) \hat \varphi (x) \le |\hat \varphi (x)|  + |x^2\hat \varphi(x)|  \le ||\varphi||_1 + ||\varphi''||_1,
\]
where the second inequality follows by applying integration by parts, and the following bound 
\[
\sum_{x\in F} |\hat \varphi(x)| \le C(||\varphi||_1 + ||\varphi''||_1),
\]
then follows by summing over $x$.
\end{proof}

We caution that the following result applies to the original trace formula, e.g., \cite[(6.25)]{GJ}. It appears in the complement of $J^G_\text{ell}(f)$ in $J^G(f)$ in Theorem \ref{main} but  not in the form that we have obtained in Theorem \ref{main2}. To treat the latter case, one would have to continue the analysis as sketched at the end of \S\ref{matz}.

\begin{cor}
The regular unipotent contribution to $J^G(f^r_s)$ is meromorphic in $s$ and holomorphic for $\mathrm{Re}(s)>1$.
\end{cor}

\begin{proof}
Following \cite[p.236]{GJ}, the expression regular unipotent contribution to the trace formula is derived from a Poisson summation formula for 
\[
F(x) = \int_{K}f\left(k^{-1}\begin{pmatrix}1 & x \\ 0 & 1\end{pmatrix}k\right)dk 
\]
as a Schwartz-Bruhat function on $\mathbb A$. Extending this via Lemma \ref{tate}, the same argument in \cite[p.236]{GJ} gives us that the unipotent contribution is indeed meromorphic in $s$ and holomorphic in Re($s)>1$. In fact, one can do better, as since the unipotent contribution essentially reduces to Tate integrals, the analysis in \cite[V.i.v]{matz0} can be applied in a similar manner, so we leave this to the reader.
\end{proof} 

Note that this simple consequence is true precisely because it is in the 1-dimensional setting, where any representation $r$ of GL(1) remains 1-dimensional, a fact which Lemma \ref{tate} can be understood as an incarnation of. In dimension 2 or more, this is no longer the case and the complexity increases along with the dimension of $r$. This is reflected by the fact that for any cuspidal automorphic representation $\pi$ of GL(2), the automorphic $L$-function $L(s,\pi,r)$ will have poles depending crucially on the representation $r$.

\subsection*{Acknowledgments} We thank Julia Gordon and Ramin Takloo-Bighash for helpful discussions.
This work is partially supported by NSF grant DMS-2212924.

\bibliography{BESL2}
\bibliographystyle{alpha}
\end{document}